\documentclass[12pt,reqno]{amsart}

\usepackage{amssymb}
\newtheorem{theorem}{Theorem}
\newtheorem{proposition}[theorem]{Proposition}
\newtheorem{lemma}[theorem]{Lemma}

\theoremstyle{remark}

\numberwithin{equation}{section}

\def\ep{\varepsilon}

\def\rk{\operatorname{rk}}
\def\Cat{\operatorname{Cat}}
\def\({\left(}
\def\){\right)}

\begin{document}

\title[Euler characteristic of generalized noncrossing partitions]{Euler characteristic of the truncated order complex of generalized noncrossing partitions}

\author[D. Armstrong and C. Krattenthaler]{D. Armstrong$^\dagger$ and C. Krattenthaler$^\ddagger$}
\address{School of Mathematics, University of Minnesota, Minneapolis, Minnesota 55455, USA. WWW: \tt http://www.math.umn.edu/\~{}armstron.}
\address{Fakult\"at f\"ur Mathematik, Universit\"at Wien, Nordbergstra{\ss}e~15, A-1090 Vienna, Austria. WWW: \tt http://www.mat.univie.ac.at/\~{}kratt.}

\thanks{$^\dagger$Research partially supported by the Austrian
Science Foundation FWF, grants Z130-N13 and S9607-N13,
the latter in the framework of the National Research Network
``Analytic Combinatorics and Probabilistic Number Theory"}



\begin{abstract}
The purpose of this note is to complete the study, begun in the first author's PhD thesis, of the topology of the poset of generalized noncrossing partitions
associated to real reflection groups. In particular, we calculate the Euler characteristic of this poset with the maximal and minimal elements deleted. 
As we show, the result on the Euler characteristic extends to generalized 
noncrossing partitions associated to well-generated complex reflection groups.
\end{abstract}

\maketitle

\section{Introduction}
\label{sec:1} 
We say that a partition of the set $[n]:=\{1,2,\ldots,n\}$ is {\sf noncrossing} if, whenever we have $\{a,c\}$ in block $A$ and $\{b,d\}$ in block $B$ of the partition with $a<b<c<d$, it follows that $A=B$. For an introduction to the rich history of this subject, see \cite[Chapter~4.1]{ArmDAA}. We say that a noncrossing partition of $[mn]$ is {\sf $m$-divisible} if each of its blocks has cardinality divisible by $m$. The collection of $k$-divisible noncrossing partitions of $[kn]$ --- which we will denote by $NC^{(m)}(n)$ --- forms a join-semilattice under the refinement partial order. This structure was first studied by Edelman in his PhD thesis; see \cite{Ede}.

Twenty-six year later, in his own PhD thesis \cite{ArmDAA}, the first author
defined a generalization of Edelman's poset to all finite 
real reflection
groups. (We refer the reader to
\cite{HumpAC} for all terminology related to 
real reflection groups.)
Let $W$ be a finite group generated by reflections in Euclidean space, and let $T\subseteq W$ denote the set of {\sf all} reflections in the group. Let $\ell_T:W\to\mathbb{Z}$ denote the word length in terms of the generators $T$. Now fix a Coxeter element $c\in W$ and a positive integer $m$. We define the set of {\sf $m$-divisible noncrossing partitions} as follows:
\begin{align} \notag
NC^{(m)}(W)=\Bigg\{ (w_0;w_1,\ldots,w_m)\in W^{m+1} &: w_0w_1\cdots w_m=c \quad\text{ and } \\ &\sum_{i=0}^m \ell_T(w_i)=\ell_T(c)\Bigg\}.
\label{eq:GNC}
\end{align}
That is, $NC^{(m)}(W)$ consists of the {\sf minimal factorizations} of $c$ into $m+1$ group elements. We define a partial order on $NC^{(m)}(W)$ by setting
\begin{multline} \label{eq:GNCord}
(w_0;w_1,\ldots,w_m) \leq (u_0;u_1,\ldots,u_m) 
\quad \text{if and only if}\quad 
\ell_T(u_i)+\ell_T(u_i^{-1}w_i) = \ell_T(w_i)\\ 
\text{for } 1\leq i\leq m.
\end{multline}
In other words, we set $(w_0;w_1,\ldots,w_m) \leq (u_0;u_1,\ldots,u_m)$ 
if for each $1\leq i\leq m$ the element $u_i$ lies on a
geodesic from the identity to $w_i$ in the Cayley graph $(W,T)$. 
We place no {\sf a priori} restriction on the elements $w_0$, $u_0$,
however it follows from the other conditions that
$\ell_T(w_0)+\ell_T(w_0^{-1}u_0)=\ell_T(u_0)$. We note that the poset is graded with rank function 
\begin{equation}\label{eq:rkm}
\rk (w_0;w_1,\ldots,w_m)= \ell_T(w_0),
\end{equation}
hence the element $(c;\varepsilon,\ldots,\varepsilon)\in W^{m+1}$ --- where
$\varepsilon\in W$ is the identity --- is the unique maximum element. There
is no unique minimum. It turns out that the isomorphism class of the poset $NC^{(m)}(W)$ is independent of the choice of Coxeter element
$c$. Furthermore, when $W$ is the symmetric group $\mathfrak{S}_n$ we recover
Edelman's poset $NC^{(m)}(n)$. 

In this note we are concerned with the {\sf order complex} $\Delta(NC^{(m)}(W))$ --- that is, with the abstract simplicial complex whose $d$-dimensional faces are the chains $\pi_0<\pi_1<\cdots < \pi_d$ in the poset $NC^{(m)}(W)$. In particular, we wish to  compute the Euler characteristic of this complex when the maximal and minimal elements of the poset have been deleted. The answer will involve the following quantity, called the {\sf positive Fu\ss--Catalan number}:
\begin{equation}\label{eq:redE}
\Cat^{(m)}_+(W):=\prod_{i=1}^n \frac{mh+d_i-2}{d_i}.
\end{equation}
Here, $n$ is the {\sf rank} of the group $W$ (the dimension of the
Euclidean space on which it acts), $h$ is the {\sf Coxeter number}
(the order of a Coxeter element), and the integers
$d_1,d_2,\ldots,d_n$ are the {\sf degrees} of $W$ (the degrees of the
fundamental $W$-invariant polynomials). Our main theorem is the
following, which settles Conjecture~3.7.9 from \cite{ArmDAA}.

\begin{theorem}\label{thm:1}
Let $W$ be a finite 
real
reflection group of rank $n$ and let $m$ be a positive integer. The order complex of the poset $NC^{(m)}(W)$ with maximal and minimal elements deleted has reduced Euler characteristic
\begin{equation} 
(-1)^n\Big(\Cat_+^{(m)}(W)-\Cat_+^{(m-1)}(W)\Big),
\end{equation}
and it is homotopy equivalent to a wedge of this many $(n-2)$-dimensional spheres.
\end{theorem}

In Section \ref{sec:2} we will collect some auxiliary results and in Section \ref{sec:3} we will prove the main theorem.

In \cite{BesDAB,BeCoAA}, Bessis and Corran have shown that the notion of
noncrossing partitions extends rather straightforwardly to well-generated
complex reflection groups. It is not done explicitly in \cite{ArmDAA},
but from \cite{BesDAB,BeCoAA} it is obvious that the definition of generalized
noncrossing partitions in \cite{ArmDAA} can be extended without any effort to 
well-generated complex
reflection groups, the same being true for many (most?) of the results from
\cite{ArmDAA} (cf.\ \cite[Disclaimer~1.3.1]{ArmDAA}). In
Section~\ref{sec:complex}, we show that the assertion in
Theorem~\ref{thm:1} on the Euler
characteristic of the truncated order complex of generalized
noncrossing partitions continues to hold for well-generated
complex reflection groups. 
We suspect that this is also true for the topology part of
Theorem~\ref{thm:1}, but what is missing here is the extension 
to well-generated complex reflection groups of the
result of Hugh Thomas and the first author \cite[Cor.~3.7.3]{ArmDAA}
that the poset of generalized noncrossing partitions associated to real
reflection groups is shellable. This extension has so far not even been
done for \cite{AtBWAA}, the special case of the poset of noncrossing
partitions.

\section{Auxiliary results}
\label{sec:2}

In this section we record some results that are needed in the proof of the main theorem.  The first result is Theorem~3.5.3 from \cite{ArmDAA}.

\begin{theorem} \label{thm:2}
The cardinality of $NC^{(m)}(W)$ is given by the {\sf Fu\ss--Catalan number for reflection groups}
\begin{equation} \label{eq:FC}
\Cat^{(m)}(W):=\prod _{i=1} ^{n}\frac {mh+d_i} {d_i},
\end{equation}
where, as before, $n$ is the rank, $h$ is the Coxeter number and the $d_i$ are the degrees of $W$. Equivalently, given a Coxeter element $c$, the number of minimal decompositions 
\begin{equation*}
w_0w_1\cdots w_m=c\quad \text{with}\quad \ell_T(w_0)+\ell_T(w_1)+\dots+\ell_T(w_m)=\ell_T(c)
\end{equation*}
is given by $\Cat^{(m)}(W)$.
\end{theorem}

Since the numbers $h-d_i+2$ are a permutation of the degrees \cite[Lemma 3.16]{HumpAC}, we have an alternate formula for the positive Fu\ss--Catalan number:
\begin{equation*}
\Cat_+^{(m)}(W)=\prod _{i=1} ^{n}\frac{mh+d_i-2} {d_i}=(-1)^n\Cat^{(-m-1)}(W).
\end{equation*}
Our next result is Theorem~3.6.9(1) from \cite{ArmDAA}.
\begin{theorem} \label{thm:3}
The total number of {\em(}multi-{\em)}chains
\begin{equation*}
\pi_1\le\pi_2\le\dots\le\pi_{l}
\end{equation*}
in $NC^{(m)}(W)$ is equal to $\Cat^{(ml)}(W)$. 
\end{theorem}

And, moreover, we have the following.

\begin{lemma} \label{lem:1}
The number of {\em(}multi-{\em)}chains $\pi_1\le\pi_2\le\dots\le\pi_{l}$ in $NC^{(m)}(W)$ with $\rk(\pi_1)=0$ is equal to $\Cat^{(ml-1)}(W)$.
\end{lemma}

\begin{proof}
If $\pi_1=(w_0^{(1)};w_1^{(1)},\ldots,w_m^{(1)})$ then the condition $\rk(\pi_1)=0$ is equivalent to $w_0^{(1)}=\varepsilon$. We note that Theorem~3.6.7 of \cite{ArmDAA}, together with the fundamental map between multichains and minimal factorizations \cite[Definition~3.2.3]{ArmDAA}, establishes a bijection between multichains $\pi_1\leq\cdots \leq\pi_l$ in $NC^{(m)}(W)$ and elements $(u_0;u_1,\ldots,u_{ml})$ of $NC^{(ml)}(W)$ for which $u_0=w_0^{(1)}$. Since $\ell_T(\varepsilon)=0$, we wish to count factorizations $u_1u_2\cdots u_{ml}=c$ in which $\ell_T(u_1)+\cdots +\ell_T(u_{ml})=\ell_T(c)$. By the second part of Theorem~\ref{thm:2}, this number is equal to $\Cat^{(ml-1)}(W)$, as desired.
\end{proof}

A stronger version of
rank-selected chain enumeration will be important in the proof of our main theorem in Section~\ref{sec:3}. Given a finite reflection group $W$ of rank $n$, let $R_W(s_1,s_2,\dots,s_l)$ denote the number of (multi-)chains
\begin{equation*}
\pi_1\le\pi_2\le\dots\le\pi_{l-1}
\end{equation*}
in $NC^{(m)}(W)$, such that $\rk(\pi_i)=s_1+s_2+\dots+s_i$, $i=1,2,\dots,l-1$, and $s_1+s_2+\dots+s_l=n$. The following lemma says that zeroes in the argument of $R_W(.)$ can be suppressed except for a zero in the first argument.

\begin{lemma} \label{lem:2}
Let $W$ be a finite real reflection group of rank $n$ and let $s_1,s_2,\dots,s_l$ be non-negative integers with $s_1+s_2+\dots+s_l=n$. Then
\begin{equation} \label{eq:RW0} 
R_W(s_1,\dots,s_i,0,s_{i+1},\dots,s_l)=
R_W(s_1,\dots,s_i,s_{i+1},\dots,s_l)
\end{equation}
for $i=1,2,\dots,l$. If $i=l$, equation \eqref{eq:RW0} must be interpreted as 
\begin{equation*}
R_W(s_1,\dots,s_l,0)=R_W(s_1,\dots,s_l).
\end{equation*}
\end{lemma}

\begin{proof}
This is obvious as long as $i<l$. If $i=l$, then, by definition,  $R_W(s_1,\dots,s_l,0)$ counts all multi-chains $\pi_1\le\pi_2\le\dots\le\pi_l$ with $\rk(\pi_i)=s_1+s_2+\dots+s_i$, $i=1,2,\dots,l$. In particular, $\rk(\pi_l)=s_1+s_2+\dots+s_l=n$, so that $\pi_l$ must be the unique maximal element $(c;\ep,\dots,\ep)$ of $NC^{(m)}(W)$. Thus we are counting multi-chains $\pi_1\le\pi_2\le\dots\le\pi_{l-1}$ with $\rk(\pi_i)=s_1+s_2+\dots+s_i$, $i=1,2,\dots,l-1$, and, again by definition, this number is given by $R_W(s_1,\dots,s_l)$.
\end{proof}

Finally we quote the version of inclusion-exclusion given in  \cite[Sec.~2.1, Eq.~(4)]{StanAP} that will be relevant to us.

\begin{proposition} \label{prop:1}
Let $A$ be a finite set and $w:A\to\mathbb C$ a weight function on $A$. Furthermore, let $S$ be a set of properties an element of $A$ may or may not have. Given a subset $Y$ of $S$, we define the functions $f_=(Y)$ and $f_\ge(Y)$ by
\begin{equation*}
f_=(Y):={\sum _{a} \,}^{\displaystyle\prime}w(a),
\end{equation*}
where $\sum{}^{\textstyle\prime}$ is taken over all $a\in A$ which have {\sf exactly} the properties $Y$, and by 
\begin{equation*}
f_\ge(Y)=\sum _{X\supseteq Y} ^{}F_=(X).
\end{equation*}
Then
\begin{equation} \label{eq:inclexcl} 
f_=(\emptyset)=
\sum _{Y\subseteq S} ^{}(-1)^{\vert Y\vert}f_\ge(Y).
\end{equation}
\end{proposition}

\section{Proof of Main Theorem}
\label{sec:3}
Let $\bf{c}$ denote the unique maximum element $(c;\varepsilon,\ldots,\varepsilon)$ of $NC^{(m)}(W)$ and let {\sf mins} denote its set of minimal elements, the cardinality of which is $\Cat^{(m-1)}(W)$. The truncated poset
\begin{equation*}
NC^{(m)}(W)\backslash\big(\{{\bf c}\}\cup \sf mins\big)
\end{equation*}
is a rank-selected subposet of $NC^{(m)}(W)$, the latter being shellable due to \cite[Cor.~3.7.3]{ArmDAA}. If we combine this observation with the fact (see \cite[Theorem~4.1]{BjoeAA}) that rank-selected subposets of shellable posets are also shellable, we conclude that  $NC^{(m)}(W)\backslash\big(\{{\bf c}\}\cup \sf mins\big)$ is shellable. Since it is known  that a pure $d$-dimensional shellable simplicial complex $\Delta$ is homotopy equivalent to a wedge of $\tilde\chi(\Delta)$  $d$-dimensional spheres  (this follows from Fact~9.19 in \cite{BjoeAB} and the fact that shellability implies the property of being homotopy-Cohen-Macaulay \cite[Sections~11.2, 11.5]{BjoeAB}),  it remains only to compute the reduced Euler characteristics $\tilde\chi(.)$ of (the order complex of) $NC^{(m)}(W)\backslash\big(\{c\}\cup \sf mins\big)$.

For a finite real reflection group $W$ of rank $n$, let us again write $R_W(s_1,s_2,\dots,s_l)$ for the number of (multi-)chains
\begin{equation*}
\pi_1\le\pi_2\le\dots\le\pi_{l-1}
\end{equation*}
in $NC^{(m)}(W)$ with $\rk(\pi_i)=s_1+s_2+\dots+s_i$, $i=1,2,\dots,l-1$, and $s_1+s_2+\dots+s_l=n$. By definition, the reduced Euler characteristic is
\begin{equation} \label{eq:1}
-1+\sum _{l=2} ^{n}(-1)^l
\underset{s_1,\dots,s_l>0}{\sum _{s_1+\dots+s_l=n} ^{}}
R_W(s_1,s_2,\dots,s_l).
\end{equation}
The sum over $s_1,s_2,\dots,s_l$ in \eqref{eq:1} could be easily calculated from Theorem~\ref{thm:3}, if there were not the restriction $s_1,s_2,\dots,s_l>0$. In order to overcome this difficulty, we appeal to the principle of inclusion-exclusion. More precisely, for a fixed $l$, in Proposition~\ref{prop:1} choose $A=\{(s_1,s_2,\dots,s_l):s_1+s_2+\dots+s_l=n\}$, $w\big((s_1,s_2,\dots,s_l)\big)=R_W(s_1,s_2,\dots,s_l)$,
and $S=\{S_i:i=1,2,\dots,l\}$, where $S_i$ is the property of an element
$(s_1,s_2,\dots,s_l)\in A$ to satisfy $s_i=0$. Then \eqref{eq:inclexcl} becomes
\begin{equation*}
 \underset{s_1,\dots,s_l>0}{\sum _{s_1+\dots+s_l=n} ^{}}
R_W(s_1,s_2,\dots,s_l)=
\sum _{I\subseteq\{1,\dots,l\}} ^{}(-1)^{\vert I\vert}
\underset{s_i=0\text{ for }i\in I}
{\underset{s_1,\dots,s_l\ge0}{\sum _{s_1+\dots+s_l=n} ^{}}}
R_W(s_1,s_2,\dots,s_l).
\end{equation*}
In view of Lemma~\ref{lem:2}, the right-hand side may be simplified, so that we obtain the equation
\begin{align*}
 \underset{s_1,\dots,s_l>0}{\sum _{s_1+\dots+s_l=n} ^{}}
R_W(s_1,s_2,\dots,s_l)&=
\underset{1\in I}{\sum _{I\subseteq\{1,\dots,l\}} ^{}}(-1)^{\vert I\vert}
{\underset{s_2,\dots,s_{l-\vert I\vert+1}\ge0}
{\sum _{s_2+\dots+s_{l-\vert I\vert+1}=n} ^{}}}
R_W(0,s_2,\dots,s_{l-\vert I\vert+1})\\
&\kern1cm
+\underset{1\notin I}{\sum _{I\subseteq\{1,\dots,l\}} ^{}}(-1)^{\vert I\vert}
{\underset{s_1,\dots,s_{l-\vert I\vert}\ge0}
{\sum _{s_1+\dots+s_{l-\vert I\vert}=n} ^{}}}
R_W(s_1,s_2,\dots,s_{l-\vert I\vert})\\
&=
\sum _{j=1} ^{l}(-1)^{j}\binom {l-1}{j-1}
{\underset{s_2,\dots,s_{l-j+1}\ge0}
{\sum _{s_2+\dots+s_{l-j+1}=n} ^{}}}
R_W(0,s_2,\dots,s_{l-j+1})\\
&\kern1cm
+\sum _{j=0} ^{l}(-1)^{j}\binom {l-1}j
{\underset{s_1,\dots,s_{l-j}\ge0}
{\sum _{s_1+\dots+s_{l-j}=n} ^{}}}
R_W(s_1,s_2,\dots,s_{l-j}).
\end{align*}
By Lemma~\ref{lem:1}, the sum over $s_2,\dots,s_{l-j+1}$ on the right-hand side is equal to\break 
$\Cat^{((l-j)m-1)}(W)$, while by Theorem~\ref{thm:3} the sum over $s_1,\dots,s_{l-j}$  is equal to\break 
$\Cat^{((l-j-1)m)}(W)$. If we substitute all this in \eqref{eq:1}, we arrive at the expression
\begin{multline} \label{eq:2}
-1+\sum _{l=2} ^{n}(-1)^l
\Bigg(
\sum _{j=1} ^{l}(-1)^j\binom {l-1}{j-1}
\Cat^{((l-j)m-1)}(W)\\
+\sum _{j=0} ^{l-1}(-1)^j\binom {l-1}{j}
\Cat^{((l-j-1)m)}(W)
\Bigg)
\end{multline}
for the reduced Euler characteristics that we want to compute. We now perform the replacement $l=j+k$ in both sums. Thereby we obtain the expression
\begin{multline} \label{eq:3}
-1-\Cat^{(-1)}(W)-\Cat^{(-m)}(W)+\Cat^{(0)}(W)\\
+\sum _{k=0} ^{n}(-1)^k
\Bigg(
\sum _{j=1} ^{n-k}\binom {j+k-1}{j-1}
\Cat^{(km-1)}(W)
+\sum _{j=0} ^{n-k}\binom {j+k-1}{j}
\Cat^{((k-1)m)}(W)
\Bigg),
\end{multline}
the various terms in the first line being correction terms that cancel terms in the sums in the second line violating the condition $l=j+k\ge2$, which is present in \eqref{eq:2}. Since we shall make use of it below, the reader should observe that, by the definition \eqref{eq:FC} of Fu\ss--Catalan numbers, both $\Cat^{(km-1)}(W)$ and $\Cat^{((k-1)m)}(W)$ are polynomials in $k$ of degree $n$ with leading coefficient $(mh)^n$.

Again by \eqref{eq:FC}, we have $\Cat ^{(-1)}(W)=0$ and $\Cat ^{(0)}(W)=1$. Therefore, if we evaluate the sums over $j$ in \eqref{eq:3} (this is a special instance of the Chu--Vandermonde summation), then we obtain the expression
\begin{align*}
-\Cat&^{(-m)}(W)+\sum _{k=0} ^{n}(-1)^k
\Bigg(\binom {n}{k+1}\Cat^{(km-1)}(W)
+\binom {n}{k}\Cat^{((k-1)m)}(W)
\Bigg)\\
&=-\Cat^{(-m)}(W)+\Cat^{(-m-1)}(W)\\
&\kern2cm
-\sum _{k=0} ^{n}(-1)^k
\binom {n}{k}\Cat^{((k-1)m-1)}(W)
+\sum _{k=0} ^{n}(-1)^k \binom {n}{k}\Cat^{((k-1)m)}(W)\\
&=-\Cat^{(-m)}(W)+\Cat^{(-m-1)}(W)-(-1)^nn!(mh)^n+(-1)^nn!(mh)^n\\
&=-\Cat^{(-m)}(W)+\Cat^{(-m-1)}(W)\\
&=-(-1)^n\Cat_+^{(m-1)}(W)+(-1)^n\Cat_+^{(m)}(W),
\end{align*}
where, to go from the second to the third line, we used the well-known fact from finite difference calculus (cf. \cite[Sec.~1.4, Eq.~(26) and Prop.~1.4.2]{StanAP}), that, for any polynomial $p(k)$ in $k$ of degree $n$ and leading coefficient $p_n$, we have
\begin{equation*}
\sum _{k=0} ^{n}(-1)^k\binom nkp(k)=(-1)^nn!p_n.
\end{equation*}

\section{The case of well-generated complex reflection groups} 
\label{sec:complex}

We conclude the paper by pointing out that our result in
Theorem~\ref{thm:1} on the Euler characteristic of the truncated poset
of generalized noncrossing partitions extends naturally to {\sf well-generated 
complex reflection groups}. We refer the reader to \cite{ShToAA,SpriAA} for all
terminology related to complex reflection groups.

Let $W$ be a finite group generated by (complex) reflections in
$\mathbb C^n$, and let $T\subseteq W$ denote the set of {\sf all}
reflections in the group. (Here, a reflection is a non-trivial element of
$GL(\mathbb C^n)$ which fixes a hyperplane pointwise and which has finite 
order.) As in Section~\ref{sec:1}, let
$\ell_T:W\to\mathbb{Z}$ denote the word length in terms of the
generators $T$. Now fix a regular element $c\in W$ in the sense of
Springer \cite{SpriAA} and a positive integer $m$. (If $W$ is a {\sf
real\/} reflection group, that is, if all generators in $T$ have order
$2$, then the notion of ``regular element" reduces to that of a
``Coxeter element.") 
As in the case of Coxeter elements, it can be shown that any two regular 
elements are conjugate to each other.
A further assumption that we need is that $W$ is {\sf well-generated},
that is, that it is generated by $n$ reflections given that $n$ is minimal
such that $W$ can be realized as reflection group on $\mathbb C^n$.

Given these extended definitions of $\ell_T$ and
$c$, we define the set of 
{\sf $m$-divisible noncrossing partitions} by \eqref{eq:GNC},
and its partial order by \eqref{eq:GNCord}, as before.
In the extension of Theorem~\ref{thm:1} to well-generated 
complex reflection groups,
we need the {\sf Fu\ss--Catalan number} for $W$, which is again
defined by \eqref{eq:FC}, where the $d_i$'s are the {\sf degrees} of 
(homogeneous polynomial generators of the invariants of) $W$, 
and where $h$ is the largest of the degrees.

\begin{theorem}\label{thm:4}
Let $W$ be a finite well-generated (complex) reflection group of rank $n$ and let $m$ be a positive integer. The order complex of the poset $NC^{(m)}(W)$ with maximal and minimal elements deleted has reduced Euler characteristic
\begin{equation} 
\Cat^{(-m-1)}(W)-\Cat^{(-m)}(W).
\end{equation}
\end{theorem}

In order to prove this theorem, we may use the proof of
Theorem~\ref{thm:1} given in Sections~\ref{sec:2} and \ref{sec:3}
essentially verbatim. The only difference is that all notions (such as
the reflections $T$ or the order $\ell_T$, for example), 
have to be interpreted in the extended sense
explained above, and that ``Coxeter element" has to be replaced by
``regular element" everywhere. In particular, the extension of
Theorem~\ref{thm:2} to well-generated complex reflection groups is 
Proposition~13.1 in \cite{BesDAB}, and the proofs of
Theorems~3.6.7 and Theorems~3.6.9(1) in \cite{ArmDAA} 
(which we used in order to establish Lemma~\ref{lem:1} respectively
Theorem~\ref{thm:3}) carry over essentially verbatim to the case of
well-generated complex reflection groups.

\end{document}